\documentclass[bj,authoryear,noshowframe]{imsart}

\usepackage{xifthen}
\usepackage{amsmath}
\usepackage{amssymb, latexsym, amsfonts, amscd, amsthm, mathrsfs}
\usepackage[usenames,dvipsnames]{color}
\usepackage[all]{xy}
\usepackage{graphicx}
\usepackage{epsfig}
\usepackage{amssymb}
\usepackage{enumerate}
\usepackage{bm}
\usepackage{todonotes}
\usepackage{bbm}
\usepackage[utf8]{inputenc}

\numberwithin{equation}{section}

\theoremstyle{plain}

\newtheorem{thm}{Theorem}[section]
\newtheorem{cor}[thm]{Corollary}
\newtheorem{lem}[thm]{Lemma}

\newtheorem{dfn}[thm]{Definition}

\theoremstyle{definition}

\newtheorem{rmk}[thm]{Remark}
\newtheorem*{rmk*}{Remark}

\startlocaldefs

\newcommand{\beq}{\begin{equation}}
\newcommand{\eeq}{\end{equation}}
	
\DeclareSymbolFont{extraup}{U}{zavm}{m}{n}
\DeclareMathSymbol{\varheart}{\mathalpha}{extraup}{86}
\DeclareMathSymbol{\vardiamond}{\mathalpha}{extraup}{87}

\newcommand{\vep}{\varepsilon}

\renewcommand{\setminus}{\backslash}

\usepackage{todonotes}
\usepackage{bbm}
\usepackage[utf8]{inputenc}

\newcommand{\eqn}[1]{\begin{equation}#1\end{equation}}
\newcommand{\eqan}[1]{\begin{align}#1\end{align}}
\newcommand{\sss}{\scriptscriptstyle}

\newcommand{\nn}{\nonumber}

\newcommand{\erdos}{Erd\H{o}s--R\'enyi random graph}
\newcommand{\e}{{\mathrm e}}

\newcommand{\DT}{{\rm DT}}

\newcommand {\convpg}{\stackrel{\sss {\mathbb P}_g}{\longrightarrow}}
\newcommand {\convpbeta}{\stackrel{\sss {\mathbb P}_{\beta}}{\longrightarrow}}

\newcommand{\oneindic}{\mathbbm{1}}
\newcommand{\indic}[1]{\oneindic_{\{#1\}}}

\newcommand{\out}[1]{}

\endlocaldefs

\begin{document}

\begin{frontmatter}

\title{Tame sparse exponential random graphs}

\runtitle{Tame sparse exponential random graphs}

\begin{aug}

\author[A]{\inits{S}\fnms{Suman}~\snm{Chakraborty}\ead[label=e1]{contact@sumanc.com}}
\author[B]{\inits{R}\fnms{Remco}~\snm{van der Hofstad}\ead[label=e2]{rhofstad@win.tue.nl}}
\author[C]{\inits{F}\fnms{Frank}~\snm{den Hollander}\ead[label=e3]{denholla@math.leidenuniv.nl}}

\address[A]{Pijnacker, Netherlands\printead[presep={,\ }]{e1}}
\address[B]{Department of Mathematics and Computer Science, Eindhoven University of Technology, Eindhoven, Netherlands\printead[presep={,\ }]{e2}}
\address[C]{Mathematical Institute, Leiden University, Leiden, Netherlands\printead[presep={,\ }]{e3}}

\end{aug}


\begin{abstract}
In this paper we obtain a precise estimate of the probability that the sparse binomial random graph contains a large number of vertices in a triangle. We compute the logarithm of this probability up to second order, which enables us to propose an exponential random graph model based on the number of vertices in a triangle. Specifically, by tuning a single parameter, we can with high probability induce any given fraction of vertices in a triangle. Moreover, in the proposed exponential random graph model we derive a large deviation principle for the number of edges. As a byproduct, we propose a consistent estimator of the tuning parameter.
\end{abstract}


\end{frontmatter}



\section{Introduction}


\subsection{Background} 

Many real-world networks are sparse, while at the same time exhibiting transitivity (also called clustering), in the sense that two neighbours of the same vertex are more likely to also be neighbours of one another (see e.g., \out{Newman}\cite{newman2009random}, \out{Rapoport}\cite{rapoport1948cycle}, \out{Serrano and Bogu{\~n}a}\cite{serrano2006clustering}, \out{Watts and Strogatz}\cite{watts1998collective}). As a result, many vertices in these graphs lie in triangles.
  
Exponential random graphs models (ERGM) are popular for modelling sparse real-world networks. Let $\mathcal{G}_n$ be the space of all simple graphs on the vertex set $[n]$, which has $2^{n\choose 2}$ elements. An ERGM can be represented by its law
	\begin{equation}
	\label{ERGM-def}
	\mathbb{P}_{\sss T}(G)=\frac{1}{Z_n(\beta)}\exp{\left(\beta T(G)\right)},\quad G\in \mathcal{G}_n,
	\end{equation}	   
where $T(G)$ is a real-valued function on the space of graphs, $\beta$ is an appropriately chosen parameter (called the inverse temperature in statistical physics), and $Z_n(\beta)$ is the normalisation constant (called the partition function). In \eqref{ERGM-def}, the random variable $T(G)$ is a sufficient statistic, in the sense that $G$ conditionally on $T(G)=t$ is {\em uniform} over all graphs $G$ with $T(G)=t$. Examples of sufficient statistics $T$ include linear combinations of subgraph counts, such as the number of edges, triangles, cycles, etc. ERGMs were first studied in \out{Holland and Leinhardt}\cite{holland1981exponential}, \out{Frank and Strauss}\cite{frank1986markov}. Several new sufficient statistics were introduced in \out{Snijders, Pattison, Robins and Handcock}\cite{snijders2006new}. 

The evaluation of the partition function $Z_n(\beta)$ is a fundamental (and often difficult) problem, and is closely related to an appropriate scaling of the inverse temperature $\beta$. In the dense regime, the first such result was obtained by \out{Chatterjee and Diaconis}\cite{chatterjee2013estimating}. While ERGMs are well-understood in the dense regime, there are hardly any results in the sparse regime (a recent result in the sparse regime appeared in \out{Mukherjee}\cite{mukherjee2020degeneracy}, and a related model called the random triangle model was studied in \out{Jonasson}\cite{jonasson1999random} and \out{H{\"a}ggstr{\"o}m and Jonasson}\cite{haggstrom1999phase}). Unfortunately, even dense exponential random graphs are problematic, as shown in \out{Bhamidi, Bresler and Sly}\cite{BhaBreSly11}: either they locally look like dense Erd\H{o}s-R\'{e}nyi random graphs, or the mixing times of Glauber dynamics or Metropolis-Hasting dynamics on them are exponentially large. Large-deviation-type estimates are the key to studying ERGMs. In dense regimes, this connection has been investigated in \out{Chatterjee and Diaconis}\cite{chatterjee2013estimating}, \out{Chatterjee and Dembo}\cite{chatterjee2016nonlinear}, \out{Bhamidi, Chakraborty, Cranmer and Desmarais}\cite{bhamidi2018weighted}.


\subsection{Goal and innovation of the present paper} 

In the present paper we continue our work in \cite{ChaHofHol21}, where we investigated sparse ERGMs based on the number of vertices $V_{T}(G)$ that participate in triangles, and we computed the correct order of scaling of $\beta$, which turns out to be of the order $\log n$. Yet, we arrived at the disappointing conclusion that either there are very few vertices in triangles in the graph (when $\beta =a \log n$ with $a<\tfrac{1}{3}$), or virtually all vertices participate in triangles (when $\beta = a\log n$ with $a>\tfrac{1}{3}$). Both are unrealistic from a practical perspective. In this paper we consider the {\em critical case}, which corresponds to $\beta=\tfrac{1}{3} \log{n} + \theta$ for arbitrary $\theta\in {\mathbb R}$, and show that now $V_{T}(G)$ scales in a non-trivial manner, leading to sparse ERGMs with a tuneable fraction of vertices in triangles. We do this via a second-order non-linear large deviation analysis of $V_{T}(G)$ in the sparse \erdos, which is novel. We further prove a large deviation principle for the number of edges, showing that the model indeed is sparse. We propose several related ERGMs based on the number of vertices in triangles, and show how to consistently estimate the parameters in such models. We reiterate that there have been numerous studies to come up with sparse ERGMs \cite{frank1986markov, handcock2003assessing, holland1981exponential, hunter2006inference, snijders2006new} with a large number of triangles but as far as we know this is the first instance where an  ERGM is rigorously shown to simultaneously satisfy the following three desirable properties: (1) With high probability, a typical outcome is sparse, i.e., the number of edges is linear in the number of vertices; (2) the number of triangles in a typical outcome is linear in the number of vertices  (and a linear number of vertices are part of some triangle); (3) consistent estimation of the parameter(s) is possible with the estimators being based on the sufficient statistic(s), the estimators are explicit, the estimation is fast, so that there is no need to resort to simulations (which is often time consuming).


\subsection{Organisation}
 
This paper is organised as follows. In Section \ref{sec-many-vertices-triangles}, we estimate the second-order of the large-deviation probabilities of the rare event that a sparse \erdos{} has a linear number of vertices in triangles, study the structure of the graph conditionally on this rare event, and provide proofs for our main results. In Section \ref{sec-ERGs}, we use these results, as well as the key insights developed in their proofs, to study exponential random graphs based on the number of vertices in triangles. We show that, for appropriate parameter choices, such models are {\em sparse}, i.e., lead to sparse exponential random graphs. In Section \ref{sec-consistent-estimation-ERGs}, we show how our main results can be used to {\em consistently} estimate the exponential random graph parameters. We close in Section \ref{sec-disc} with a discussion and a list of open problems.


\section{Large number of vertices in triangles}
\label{sec-many-vertices-triangles}

Let $G=(V(G),E(G))$ be a graph with vertex set $V(G)$ and edge set $E(G)$. Let $V_{T}(G)$ be the number of vertices that are part of a triangle in $G$. Throughout this paper, we write $G_n$ for the \erdos{} $G(n,p_n)$, where $p_n=\lambda/n$. In this section, we study the probability that $V_{T}(G_n)$ is of order $n$, as well as the structure of the graph conditionally on this event. 

Our main results, Theorems \ref{thm:ldp_vertices_in_triangle}--\ref{thm-conc-edges}, are stated in Section~\ref{sec:thms}. The upper bound of Theorem \ref{thm:ldp_vertices_in_triangle} is proved in Section~\ref{sec:ub}, the lower bound in Section~\ref{sec:lb}. In Section~\ref{sec:trianglesdisjoint} we use the technique in the proof of Theorem \ref{thm:ldp_vertices_in_triangle} to show that, conditionally on the graph having many vertices in triangles, most of the triangles are disjoint, as stated in Theorem \ref{structure:disjoint triangles}. In Section~\ref{sec:sparse} the graph is shown to be sparse, as stated in Theorem \ref{thm-conc-edges}. 


\subsection{Main theorems}
\label{sec:thms}

Our first main result is an upper tail estimate for the sparse binomial random graph:

\begin{thm}[Large deviations for number of vertices in triangles]
\label{thm:ldp_vertices_in_triangle}
Let $G_n$ be the \erdos{} with $p_n=\lambda/n$ and $\lambda \in (0,\infty)$. Then
	\begin{equation}
	\label{eqn:monfeb21426pm}
	\mathbb{P}(V_{T}(G_n)  \geq q) 
	= \frac{n(n-1) \times\cdots\times (n-q+1)} {(3!)^{q/3}(q/3)!}\, p_n^q\, \e^{o(n^{19/20})}.
	\end{equation}
In particular, for $a \in (0,1)$, 
	\begin{equation}
	\label{eqn:monfeb21427pm}
	\begin{aligned}
	&\log{\mathbb{P}(V_{T}(G_n)  \geq an)}\\ 
	&\qquad = -n(1-a) \log(1-a) - (\tfrac13 an) \log(\tfrac13 an) - an (\tfrac23  +\tfrac13 \log{6})
	+ an \log{\lambda} + o(n^{19/20}).
	\end{aligned}
	\end{equation}
\end{thm}

\noindent
Theorem \ref{thm:ldp_vertices_in_triangle} extends \cite[Theorem 1.8]{ChaHofHol21} from first- to second-order large deviations. The second-order term will be important to settle a number of open problems, including the computation of the number of edges in the graph conditionally on the graph having many vertices in triangles, as well as to prove that most of the triangles are actually disjoint. 

Our second main result says that if $G_n$ is conditioned to have of the order $n$ vertices in a triangle, then almost all the triangles are vertex disjoint. To state the result, we let $\DT_n(G_n)$ be the maximum number of vertex-disjoint triangles in the graph $G_n$:

\begin{thm}[Most triangles are disjoint conditionally on $V_T(G_n)=an$]
\label{structure:disjoint triangles}
For any $\varepsilon>0$,
	\begin{equation}
	\label{eqn:thumar301158am}
	\mathbb{P}\big(\DT_n(G_n)\leq \tfrac13(a-\varepsilon)n \,|\, V_T(G_n)\geq an\big) 
	\leq  \exp\left(-\tfrac16\varepsilon n\log{n} + Cn\right),
	\end{equation} 
where $C$ is a constant that is independent of $\varepsilon$.
\end{thm}

Our third main result is the following concentration property on the number of edges in the graph:

\begin{thm}[Concentration of edges conditionally on $V_T(G_n)\geq an$]
\label{thm-conc-edges}
Conditionally on $V_T(G_n)\geq an$, the number of edges in $G_n$ is concentrated around $(a+\lambda/2)n$.
\end{thm}


\subsection{Upper bound in Theorem \ref{thm:ldp_vertices_in_triangle}}
\label{sec:ub}

A central definition in the proof is that of a \emph{$q$-basic graph} and its configuration:

\begin{dfn}[$q$-basic graphs]
\label{def-q-basic}
$\mbox{}$
{\rm A subgraph $G \subseteq K_n$ is called \emph{$q$-basic} if $V_{\sss T}(G) = q$ and $V_{\sss T}(G\setminus e) < q$ for all edges $e \in G$. Here, $G\setminus e$ denotes the graph with the edge $e$ removed.}
\end{dfn}

\begin{figure}
\includegraphics[scale=0.3]{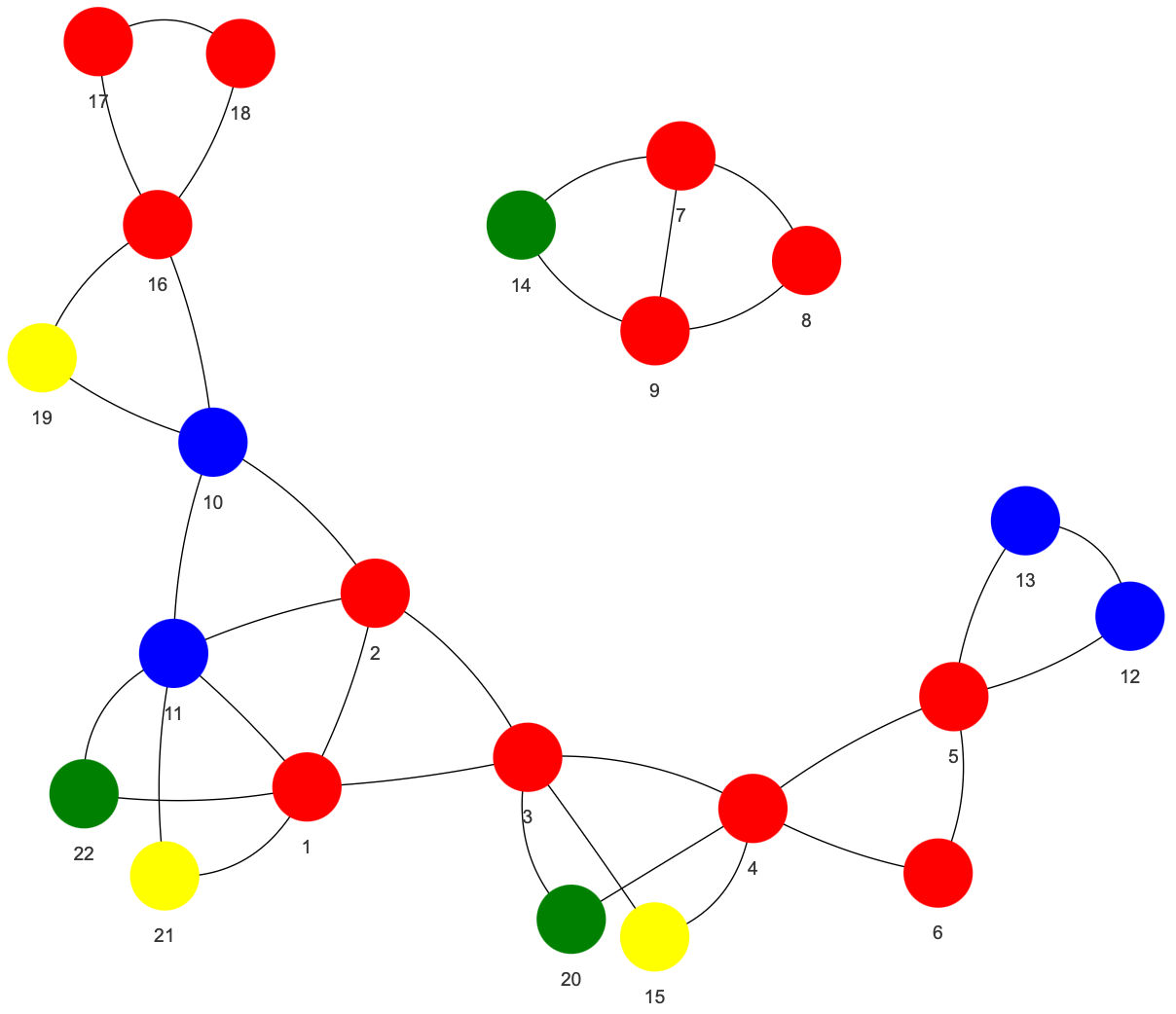}
\caption{A stylized example of the graph decomposition of a $q$-basic graph. $V_1$ consists of the red vertices, $V_2$ of the blue vertices, $V_{31}$ of the yellow vertices, and $V_{32}$ of the green vertices, and $\mathcal{W}=\{(1,11),(3,4),(10,16)\}$.} 
\label{fig:graph_decomposition}
\end{figure}

\noindent
In the following, we derive a precise estimate for the number of subgraphs of $K_n$ containing a specified number of vertices, denoted by $q$, that belong to at least one triangle. To achieve this, we count the number of $q$-basic subgraphs of $K_n$. We will come back to this counting after proving the main technical lemma that facilitates the analysis (Lemma \ref{lem:fridec20454pm} below). A weaker version of this lemma was used in the proof of \cite[Theorem 1.8]{ChaHofHol21}. The following lemma is more precise, and is of independent interest. We use the following standard notation: if $G$ is a graph with vertex set $V$, then, for $U\subseteq V$, $G[U]$ denotes the subgraph induced in $G$ by $U$. For a fixed graph $H$ on $h$ vertices, we say that a graph $G$ on $n$ vertices has an $H$-factor if $G$ contains $n/h$ vertex disjoint copies of $H$. For example, if $H$ is just an edge $H= K_2$, then an $H$-factor is a perfect matching in $G$.

\begin{lem}[$q$-basic graph decomposition]
\label{lem:fridec20454pm}
We can partition the vertices of a $q$-basic graph $G=(V,E)$ into (not necessarily uniquely) sets $(V_1,V_2,V_{31}, V_{32})$ in such a way that the following conditions are met: 
\begin{itemize}
\item[(1)]  
$G[V_1]$ has a $K_3$-factor, while $G[V\setminus V_1]$ has no triangles. (This $K_3$ factor contributes $|V_1|$ edges. $G[V_1]$ may contain additional edges that connect two disjoint triangles in this $K_3$-factor, and we construct a set called $\mathcal{W}$ consisting of these additional edges.)
\item[(2)] $G[V_2]$ has a $K_2$-factor, where the ends of the disjoint edges of this $K_2$-factor have a common neighbor in $V_1$, $G[V_2]$ does not have any additional edge, and $G[V\setminus (V_1\cup V_2)]$ is an independent set. (Note that this $K_2$ factor contributes $|V_2|/2$ edges. The ends of the disjoint edges in $G[V_2]$ have a common neighbor in $V_1$ and this contribute another $|V_2|$ edges. There might still be more edges between $V_1$ and $V_2$, and we add them to the set $\mathcal{W}$.)
\item[(31)] 
Each vertex in $V_{31} \subseteq V\setminus (V_1\cup V_2)$ is a common neighbor of endpoints of some edge in $\mathcal{W}$, and endpoints of each edge in $\mathcal{W}$ has exactly one common neighbor in $V_{31}$. (Note that the edges in $\mathcal{W}$ and the edges between the endpoints of edges in $\mathcal{W}$ and $V_{31}$ contributes another $3|V_{31}|$ edges.)
\item[(32)] 
Each vertex in $V_{32}=V\setminus (V_1\cup V_2\cup V_{31})$ is a common neighbor of endpoints of an edge in $G[V_1]$ or of an edge between $V_1$ and $V_2$. (Note that every vertex $V_{32}$ contributes two edges as they are common neighbor of endpoints of an edge in $G[V_1]$ or of an edge between $V_1$ and $V_2$, and therefore they contribute  $2|V_{31}|$ edges in total.)
\end{itemize}
\end{lem}

\begin{proof}
The decomposition is done in a greedy manner. See Figure \ref{fig:graph_decomposition} for an example of the partition.\par
In the first step we extract disjoint triangles (in an arbitrary order) from $G$, and we stop when we are unable to extract more triangles (thus, $G[V\setminus V_1]$ is triangle-free). The vertex set of the extracted disjoint triangles is denoted by $V_1$. By construction, $G[V_1]$ has a $K_3$ factor. Note that there could be edges in $G[V_1]$ that are not used in the formation of the $K_3$ factor. We construct a set called $\mathcal{W}$ consisting of these additional edges.
\par
We first claim that $G[V\setminus V_1]$ does not contain a path of length three. We prove this claim by contradiction. Suppose that $G[V\setminus V_1]$ has a path of length three consisting of the edges $(u_1,u_2),(u_2, u_3), (u_3,u_4)$. We show that even if we remove the edge $(u_2,u_3)$, then also all the vertices in $V$ remain part of some triangle, which contradicts the fact that $G$ is $q$-basic. First note that every vertex in $V_1$ is part of some triangle in $G[V_1]$ as $G[V_1]$ has a $K_3$-factor. Secondly, since $G$ is $q$-basic, the edges $(u_1,u_2)$ and $(u_3,u_4)$ must be part of some triangles of the form $(u_1,u_2,o_1)$ and $(u_3,u_4,o_2)$, where $o_1,o_2 \in V_1$. Note that $o_1,o_2$ cannot be outside $V_1$, as $G[V\setminus V_1]$ is triangle-free. Therefore $u_2$ and $u_3$ are part of some triangles that do not involve the edge $(u_2,u_3)$. Finally, take a vertex $w \in V\setminus (V_1 \cup \{u_2,u_3\})$. Note that, although $w$ must be part of some triangle, it cannot be part of triangle of the form $(u_2,u_3,w)$, as this triangle lies outside $V_1$. Therefore, even if we remove the edge $(u_2,u_3)$, all vertices of $G$ remain part of some triangle, proving our claim.
\par
Next we extract disjoint edges from $G[V\setminus V_1]$, we stop when there are no more disjoint edges left, and we denote the vertex set of these disjoint edges by $V_2$. Thus, by construction, $G[V\setminus (V_1\cup V_2)]$ is an independent set and $G[V_2]$ has a $K_2$ factor. Since $G$ is a $q$-basic graph, the edges in $G[V_2]$ must be part of some triangle in $G$, and therefore the endpoints of each edge in $G[V_2]$ must have some common neighbor in $V_1$. This common neighbor cannot be in $V\setminus V_1$, as that would create a triangle in $G[V\setminus V_1]$. The $K_2$ factor of $G[V_2]$ contributes $|V_2|/2$ edges in $G[V_2]$, and the endpoints of each of these edges have a common neighbor in $V_1$, which contributes $|V_2|$ edges between $V_1$ and $V_2$. If there are more edges between $V_1$ and $V_2$, then we add them to the set $\mathcal{W}$.
\par
Since $G$ is $q$-basic, the endpoints of each edge in $\mathcal{W}$ must have at least one common neighbor in $V\setminus V_1 \cup V_2$ (note that this common neighbor cannot be in $V_1 \cup V_2$, as all the vertices in $V_1 \cup V_2$ are already part of some triangle that do not contain edges from $\mathcal{W}$). For each edge in $\mathcal{W}$ we pick exactly one such vertex from $V\setminus V_1 \cup V_2$ that is a common neighbor of the endpoints of that edge, and construct the set $V_{31}$ with these vertices (if there are more than one, then we arbitrarily pick one). Note that $|V_{31}|=|\mathcal{W}|$ and this contributes $3|V_{31}|$ edges ($|V_{31}|$ edges in $|\mathcal{W}|$ and endpoints of edge in $\mathcal{W}$ have exactly one common neighbor in $V_{31}$, which contributes $2|V_{31}|$ edges).
\par
Finally, $V_{32}= V\setminus V_1\cup V_2 \cup V_{31}$. They are common neighbor of endpoints of an edge between $V_1$ and $V_2$ or common neighbor of an edge in $V_1$. Therefore they contribute $2|V_{32}|$ edges (the edges between $V_1$ and $V_2$ and the edges in $V_1$ were already counted before).
\end{proof}

\noindent
The three steps in the proof of Lemma \ref{lem:fridec20454pm} are illustrated in Figure \ref{fig:graph_decomposition}. 

We next define $q$-basic graph with a given configuration. This will help us count the number of edges in such graphs. 

\begin{dfn}[$q$-basic graphs with configuration]
\label{def-q-basic_config}
{\rm A $q$-basic graph $G$ has an $(\ell_1,\ell_2,\ell_{31},\ell_{32})$ configuration if there is a graph partition $(V_1,V_2, V_{31}, V_{32})$ as described in Lemma \ref{lem:fridec20454pm} with $|V_i|=\ell_i$ for $i=1,2, 31, 32$. (Note that $q=\ell_1+\ell_2+\ell_{31}+ \ell_{32}$.).}
\end{dfn}

It is clear that the number of edges in a $q$-basic graph with an $(\ell_1,\ell_2,\ell_{31},\ell_{32})$ configuration is equal to $\ell_1 + \tfrac32 \ell_2 + 3 \ell_{31} + 2\ell_{32}$. The number of $q$-basic subgraphs of the complete graph $K_n$ with an $(\ell_1,\ell_2,\ell_{31},\ell_{32})$ configuration is at most 
	\begin{equation}
	\label{e:form1}
	\frac{n(n-1) \times\cdots\times (n-q +1)} {(3!)^{\ell_1/3} \, (\ell_1/3)! \, (2!)^{{\ell_2/2}} \, (\ell_2/2)! \, \ell_3!}
	\times \ell_1^{\ell_2/2} \times {\ell_1^2/2 
	+ \ell_1 \ell_2 \choose \ell_{31}} \times {\ell_1 + \ell_2 + \ell_{31} \choose \ell_{32}},
	\end{equation}
where $\ell_3 = \ell_{31} + \ell_{32}$. Indeed, the factor $n!/(n-q)!$ equals the number of ways we can choose {an ordered set of $q$ distinct vertices out of $n$ vertices}. However, permutations of the 3 vertices per triangle, as well as the collection of 3 vertices in the $\ell_1/3$ triangles, leaves the $q$-basic graph unchanged, so we have to divide by $(3!)^{\ell_1/3}\,(\ell_1/3)!$. Similarly, permutations of the pairs of vertices in the $\ell_2$ edges in $V_2$, and the $\ell_2$ edges, leave the $q$-basic graph unchanged, so we have to divide by $(2!)^{\ell_2/2}\,(\ell_2/2)!$. In the same vein, we have to divide by $\ell_3!$. This gives us the number of ways to partition the vertex set into $(V_1,V_2,V_{31}, V_{32})$. Per edge in $V_2$, we have $\ell_1$ choices for the common neighbour of its endpoints, which gives rise to $\ell_1^{\ell_2/2}$ possibilities. Per vertex in $V_{31}$, we have to choose a pair of vertices in $V_1$, or in $V_1$ and $V_2$, which gives rise to ${\ell_1^2/2 + \ell_1 \ell_2 \choose \ell_{31}}$ possibilities. Per vertex in $V_{32}$, we need to choose an edge in a triangle in $V_1$, of which there are $\ell_1$, or an edge between a vertex in $V_2$ and a vertex in $V_1$ that is used to create the triangle containing the edge in $V_2$, of which there are $\ell_2$, or an edge between $V_2$ and $V_1$ that is used to construct $V_{31}$, of which there are $\ell_{31}$. This gives rise to ${\ell_1+ \ell_2+\ell_{31} \choose \ell_{32}}$ possibilities.

We estimate $\tfrac12 \ell_1^2 + \ell_1 \ell_2 \leq q^2 \leq n^2$ and $\ell_1+ \ell_2 + \ell_{31} \leq q \leq n$, to get that \eqref{e:form1} is bounded above by
	\begin{equation}
	\label{eqn:815am01sep22}
	\frac{n(n-1) \times\cdots\times (n-q +1)} {(3!)^{\ell_1/3} \, (\ell_1/3)! \, (2!)^{{\ell_2/2}} \, (\ell_2/2)! \, \ell_3!} 
	\times \exp\big((\ell_2/2)\log{n} + 2\ell_{31}\log{n} + \ell_{32} \log{n}\big).
	\end{equation}
We estimate
	\begin{equation}
	\label{eq:81401sep22}
	\begin{aligned}
	\mathbb{P}(V_{T}(G_n) \geq k_n) 
	&= \sum_{q\geq k_n} \mathbb{P}(V_{T}(G_n) = q)\\
	&\leq \sum_{q\geq k_n} \mathbb{P}\big(G_{n,p_n} \text{ contains a }q \text{-basic subgraph}\big) \\
	& \leq \sum_{q\geq k_n} \sum_{\ell_1 + \ell_2 + \ell_{31} + \ell_{32} = q}
	\mathbb{P}\big(G_{n,p_n} \text{ contains a }q \text{-basic } (\ell_1, \ell_2, \ell_{31}, \ell_{32}) \text{ subgraph}\big) \\
	&\leq \sum_{q\geq k_n}\sum_{\ell_1 + \ell_2 + \ell_{31} + \ell_{32} = q}
	p_n^{\ell_1 + 3\ell_2/2 + 3\ell_{31} + 2\ell_{32}} \, \# \{q\text{-basic } (\ell_1, \ell_2, \ell_{31}, \ell_{32}) 
	\text{ subgraph}\} \\
	&= \sum_{q\geq k_n} p_n^{q} \sum_{\ell_1 + \ell_2 + \ell_{31} + \ell_{32} = q} 
	p_n^{\ell_2/2 + 2\ell_{31} + \ell_{32}} \, 
	\# \{q\text{-basic } (\ell_1, \ell_2, \ell_{31}, \ell_{32}) \text{ subgraphs}\}.
	\end{aligned}
	\end{equation}
Substituting \eqref{eqn:815am01sep22} into \eqref{eq:81401sep22}, we get 
	\begin{equation}
	\label{eq:82001sep22}
	\begin{aligned}
	\mathbb{P}(V_{T}(G_n) \geq k_n) \leq n^5 \!\!\!
	\max_{\ell_1 + \ell_2 + \ell_{31} + \ell_{32}\geq k_n} {\lambda}^{\ell_2/2 + 2\ell_{31} + \ell_{32}}
	\frac{n(n-1) \times\cdots\times (n-q +1)} {(3!)^{\ell_1/3}(\ell_1/3)! (2!)^{{\ell_2/2}}(\ell_2/2)!\ell_3!}.
	\end{aligned}
	\end{equation}
We are left with maximising 
	\begin{equation}
	\label{e:form2}
	\frac{{\lambda}^{\ell_2/2 + 2\ell_{31} + \ell_{32}}} {(3!)^{\ell_1/3}\,(\ell_1/3)!\,(2!)^{{\ell_2/2}}\,(\ell_2/2)!\,\ell_3!},
	\end{equation}
subject to $\ell_1+\ell_2+\ell_{31}+\ell_{32}=q$. Using Stirling's approximation $n! = n\log{n}-n+O(\log{n})$, we get that \eqref{e:form2} equals
	\begin{equation}
	\label{e:form3}
	\begin{aligned}
	&\exp\Big((\ell_2/2 + 2\ell_{31} + \ell_{32})\log\lambda - (\ell_1/3)\log(\ell_1/3) + (\ell_1/3)(1-\log{3!})\\
	&\qquad - (\ell_2/2) \log(\ell_2/2) + (\ell_2/2)(1-\log{2!}) - \ell_3\log{\ell_3} + \ell_3 + O(\log(\ell_1\vee \ell_2 \vee \ell_3))\Big).
	\end{aligned}
	\end{equation}
We can trivially upper bound \eqref{e:form3} by
	\begin{equation}
	\label{pen-final-bound-UB}
	\begin{aligned}
	&\exp\Big((\ell_2/2 + 2\ell_{3}) (\log\lambda)_+ - (\ell_1/3) \log(\ell_1/3) + (\ell_1/3) (1-\log{3!})\\
	&\qquad - (\ell_2/2)\log(\ell_2/2) + (\ell_2/2) (1-\log{2!})- \ell_3\log{\ell_3}  + \ell_3 + O(\log(\ell_1\vee \ell_2 \vee \ell_3)\Big),
	\end{aligned}
	\end{equation}
where $x_+=\max\{x,0\}$.

In order to investigate this bound further, we state a technical lemma:

\begin{lem}[Variational problem]
\label{lem:thudec1402pm}
For large enough $q$, the minimum value of 
	\begin{equation}
	f(x_1,x_2,x_3) = \tfrac13 x_1 \log{x_1} + \tfrac12 x_2 \log{x_2} + x_3 \log{x_3} 
	+ c_1x_1+c_2x_2+c_3x_3,
	\end{equation}
subject to $x_1+x_2+x_3=q$, is attained when $x_1 \geq q-q^{9/10}$ and $x_2, x_3 \leq q^{9/10}$. In particular, there exists a $C>0$ such that
	\begin{equation}
	\min_{x_1+x_2+x_3=q}f(x_1,x_2,x_3) \geq  \tfrac13 q\log{q}+ c_1q- Cq^{9/10}\log{q}.
	\end{equation}
\end{lem}

\begin{proof}
We use the Lagrange multiplier method. Define
	\begin{equation}
	\Lambda = f(x_1,x_2,x_3) + \mu (q-x_1-x_2-x_3).
	\end{equation}
The partial derivatives of $\Lambda$ with respect to $x_1,x_2,x_3$ are
	\begin{equation}
	\begin{aligned}
	&\frac{\partial \Lambda}{\partial x_1} = \tfrac{1}{3} \log{x_1} + \tfrac{1}{3} + c_1 - \mu, \\
	&\frac{\partial \Lambda}{\partial x_2} = \tfrac{1}{2} \log{x_2} + \tfrac{1}{2} + c_2 - \mu, \\
	&\frac{\partial \Lambda}{\partial x_3} = \log{x_3} + 1 + c_3 - \mu.
	\end{aligned}
	\end{equation}
Setting these equal to zero, we get the following relation for the Lagrange multiplier $\mu$:
	\begin{equation}
	\exp(3\mu-3c_1-1)\Big[1+\exp(-\mu+3c_1-2c_2)+ \exp(-2\mu+3c_1-2c_3)\Big] = q.
	\end{equation}
The solution is bounded above by $\mu_U$ and bounded below by $\mu_L$ for large enough $q$, where $\mu_U$ and $\mu_L$ are solutions of the equations
	\begin{equation}
	\exp(3\mu_U-3c_1-1) = q, \qquad \exp(3\mu_L-3c_1-1) = q-q^{9/10}.
	\end{equation}
	\end{proof}

By Lemma \ref{lem:thudec1402pm}, with $x_1=\ell_1, x_2=\ell_2, x_3=\ell_{31}+\ell_{32}, c_1=(1+\log{3}-\log{3!})/3$, for large enough $q$, \eqref{pen-final-bound-UB} is bounded above by 
	\begin{equation}
	\label{eqn:thudec221128pm}
	\exp\left(-\tfrac13 q\log(\tfrac13 q) + \tfrac13 (1-\log{3!}) q + Cq^{9/10}\log{q}\right).
	\end{equation}
This completes the proof of the upper bound in Theorem \ref{thm:ldp_vertices_in_triangle}.\qed


\subsection{Lower bound in Theorem \ref{thm:ldp_vertices_in_triangle}}
\label{sec:lb}

In this section, we prove the lower bound in Theorem \ref{thm:ldp_vertices_in_triangle}.  We abbreviate $s_n=an$, and assume that $3|an$. Then we have the following lower bound:
	\begin{equation}
	\label{eqn:244pm02may22}
	\begin{aligned}
	&\mathbb{P}(V_{ T}(G_{n,p_n}) \geq an) \geq \frac{n(n-1) \times\cdots\times (n-s_n+1)}{(3!)^{s_n/3}(s_n/3)!} \\
	&\qquad\times\mathbb{P}(123, 456, \ldots, s_{n-2}s_{n-1}s_n\;\text{ form triangles and the rest 
	of the graph is triangle free}).
	\end{aligned}
	\end{equation}
Indeed, the prefactor counts the number of ways to choose $s_n$ triples. The event that these triples form triangles and the event that the rest of the graph is triangle free are independent, and each has the same probability as when the triples are 123, 456, etc. Next, we rewrite	
	\eqan{
	&\mathbb{P}(123, 456,\ldots, s_{n-2}s_{n-1}s_n\;\text{ form triangles and the rest of the graph is triangle free})\\
	&\quad =p_n^{s_n} \mathbb{P}(\text{the rest of the graph is triangle free}\mid 123, 456,\ldots, s_{n-2}s_{n-1}s_n\text{ form triangles}).   \nn
	}
Define $B_n$ to be the set of edges that are used to form the triangles $R_n=\{123, 456,\ldots, s_{n-2}s_{n-1}s_n\}$. Clearly, $|B_n|=s_n$. Conditionally on all the edges in $B_n$ being present, the number of triangles in the graph $G_n$ excluding $R_n$ can be written as
	\begin{equation}
	\label{eqn:308pm02may22}
	W_n = \sum_{\{(i,j,k)\colon 1\leq i<j<k \leq n\}\setminus R_n} a_{ij}a_{jk}a_{ki},
	\end{equation}
where $(a_{ij})_{1 \leq i < j \leq n}$ is the adjacency matrix of $G_n$. Write
	\eqn{
	\{W_n=0\}=\cap_{\Delta} \{\Delta \text{ not present}\},
	}
where the intersection runs over all triangles $\Delta$ except for $123,456,\ldots, s_{n-2}s_{n-1}s_n$. Note that even under the conditioning the edges that are not in the triangles $123,456,\ldots, s_{n-2}s_{n-1}s_n$ are independent. Since the events $\{\Delta \text{ not present}\}$ are all decreasing, we can use Harris inequality \cite{Harr60} to get  
	\begin{align}
	&\mathbb{P}(W_n=0\mid 123, 456,\ldots, s_{n-2}s_{n-1}s_n\text{ form triangles})\\
	&\qquad\geq \prod_{\Delta} \mathbb{P}(\Delta \text{ not present}\mid 123, 456,\ldots, s_{n-2}s_{n-1}s_n\text{ form triangles})\\
	&\qquad= \prod_{\Delta \text{ contains an edge from }B_n } \mathbb{P}(\Delta \text{ not present}\mid 123, 456,\ldots, s_{n-2}s_{n-1}s_n\text{ form triangles}) \\
	&\qquad \times \prod_{\Delta \text{ does not contains an edge from }B_n } \mathbb{P}(\Delta \text{ not present}\mid 123, 456,\ldots, s_{n-2}s_{n-1}s_n\text{ form triangles}) \\
	&\qquad  \geq \left(1-\frac{\lambda^2}{n^2}\right)^{ns_n} \times  \left(1-\frac{\lambda^3}{n^3}\right)^{{n \choose 3}} \geq 0.9 \exp\left(-a\lambda^2-\lambda^3/6\right),
	\end{align}
for all $n$ large enough when $s_n=an$ for a constant $a\in (0,1)$.
\smallskip\noindent

This proves that 
	\eqn{
	\mathbb{P}(V_{ T}(G_n)\geq an) \geq \frac{n(n-1) \times\cdots\times (n-s_n+1)}{(3!)^{s_n/3}(s_n/3)!}p_n^{s_n} c,
	}
for some $c>0$, as required.\qed

\begin{rmk}
{\rm Note that the conclusion of Theorem \ref{thm:ldp_vertices_in_triangle} is meaningful when $q\geq Cn^{19/20}$ for some constant $C>0$. We are interested in estimating $\log{\mathbb{P}(V_{T}(G_n) \geq q)}$ when $q\geq Cn^{19/20}$. In this regime, both $\mathbb{P}(V_{T}(G_n)  = q)$ and $\mathbb{P}(V_{T}(G_n) \geq q)$ satisfy the estimate given in the right-hand side of \eqref{eqn:monfeb21426pm}, as the extra contribution of the latter gets absorbed into the error term $\e^{o(n^{19/20})}$.}
\end{rmk}


\subsection{Most triangles are disjoint: proof of Theorem \ref{structure:disjoint triangles}}
\label{sec:trianglesdisjoint}

In this section, we use the technique developed for the proof of the upper bound in Theorem \ref{thm:ldp_vertices_in_triangle} to prove Theorem \ref{structure:disjoint triangles}. For this we wish to evaluate
	\begin{equation}
	\label{eqn:thumar301144am}
	\mathbb{P}\big(\DT_n(G_n)\leq \tfrac13(a-\varepsilon)n \,|\, V_T(G_n)\geq an\big) 
	= \frac{\mathbb{P}(\DT_n(G_n))\leq \tfrac13(a-\varepsilon)n, V_T(G_n)\geq an) }{\mathbb{P}(V_T(G_n) \geq an)}.
	\end{equation}
Note that for the denominator we can use the estimate in Theorem~\ref{thm:ldp_vertices_in_triangle}. Obtaining an estimate for the numerator is more subtle, and for that we again use our earlier decomposition argument. The crucial observation is the following: the event $\{\DT_n(G_n)\leq \frac13(a-\varepsilon)n, V_T(G_n)\geq an\}$ implies that $G_n$ admits an $(\ell_1,\ell_2,\ell_{31},\ell_{32})$ decomposition with $\ell_1 \leq (a-\varepsilon)n$. We obtain the following upper bound for any $\varepsilon>0$:
	\eqan{
	&\mathbb{P}\big(\DT_n(G_n)\leq {\tfrac13}(a-\varepsilon)n,V_{T}(G_n)  \geq an\big)\\
	&\qquad = \sum_{q\geq an} \mathbb{P}\big(\DT_n(G_n)\leq \tfrac13(a-\varepsilon)n,V_{T}(G_n) ={q}\big)\nn\\
	&\qquad \leq  \sum_{q\geq an} \mathbb{P}\big(G_n \text{ contains a }q \text{-basic subgraph with }
	\ell_1 \leq (a-\varepsilon)n\big)\nn\\
	&\qquad \leq \sum_{q\geq an} \sum_{\ell_1+\ell_2+\ell_{31}+\ell_{32}=q}
	\mathbb{P}\big(G_n \text{ contains a } q \text{-basic }\ell_1, \ell_2, \ell_{31}, \ell_{32} 
	\text{ subgraph with }\ell_1 \leq (a-\varepsilon)n\big).\nn
	}
This is bounded from above by
	\begin{equation}
	\begin{aligned}
	&\sum_{q\geq an} \sum_{\ell_1+\ell_2+\ell_{31}+\ell_{32}=q} {p_n}^{\ell_1 +3\ell_2/2 +3\ell_{31}+2\ell_{32}} 
	\# \{q\text{-basic }(\ell_1, \ell_2, \ell_{31}, \ell_{32}) \text{ subgraph with }\ell_1 \leq (a-\varepsilon)n\} \\
	&= \sum_{q\geq an} {p_n}^{q}\sum_{\ell_1+\ell_2+\ell_{31}+\ell_{32}=q} {p_n}^{\ell_2/2 +2\ell_{31}+\ell_{32}} 
	\# \{q\text{-basic } (\ell_1, \ell_2, \ell_{31}, \ell_{32}) \text{ subgraph with } \ell_1 \leq (a-\varepsilon)n\}.
	\end{aligned}
\end{equation}
Next, we use \eqref{eqn:815am01sep22} to obtain
	\begin{equation}
	\label{eqn:thumar301135am}
	\begin{aligned}
	&\mathbb{P}\big(\DT_n(G_n)\leq \tfrac13(a-\varepsilon)n,V_{T}(G_n) \geq an\big) \\ 
	&\leq n^4 \sum_{q\geq an} {p_n}^{q}\max_{\ell_1+\ell_2+\ell_{31}+\ell_{32}=q, \ell_1 
	\leq (a-\varepsilon)n} {\lambda}^{\ell_2/2 + 2\ell_{31} + \ell_{32}}
	\frac{n(n-1) \times\cdots\times (n-q +1)}{(3!)^{\ell_1/3} (\ell_1/3)!(2!)^{{\ell_2/2}}(\ell_2/2)! \ell_3!}.
	\end{aligned}
	\end{equation}
Note that an identical calculation as for Lemma \ref{lem:thudec1402pm} yields that 
	\begin{equation}
	\label{eqn:thumar1133am}
	\begin{aligned}
	&\max_{\ell_1+\ell_2+\ell_{31}+\ell_{32}=q, \ell_1 \leq \tfrac13(a-\varepsilon)n} 
	{\lambda}^{\ell_2/2 +2\ell_{31}+\ell_{32}}\frac{n(n-1) \times\cdots\times (n-q +1)} {(3!)^{\ell_1/3}
	(\ell_1/3)!(2!)^{{\ell_2/2}}(\ell_2/2)!\ell_3!} \\
	&\qquad\qquad\qquad \leq \exp\left(-\tfrac13(a-\varepsilon)n \log{n} -\tfrac12(q-a+\varepsilon)n\log{n} +O(n)\right).	
	\end{aligned}
	\end{equation}
Therefore, inserting \eqref{eqn:thumar1133am} into \eqref{eqn:thumar301135am}, we get
	\begin{equation}
	\label{eqn:thumar301137am}
	\begin{aligned}
	&\mathbb{P}\big(\DT_n(G_n)\leq\tfrac13(a-\varepsilon)n,V_{T}(G_n) \geq an\big) \\
	&\qquad \leq \exp\left(-\tfrac13(a-\varepsilon)n\log{n} -\tfrac12(q-a+\varepsilon)n\log{n} +O(n)\right).
	\end{aligned}
	\end{equation}
Finally, insert the estimates from Theorem~\ref{thm:ldp_vertices_in_triangle} and \eqref{eqn:thumar301137am} into \eqref{eqn:thumar301144am}, to get the claim. The argument even shows that $\mathbb{P}\big(\DT_n(G_n)\leq\tfrac13(a-\varepsilon)n,V_{T}(G_n) \geq an\big)$ decays like $\e^{-\delta n\log{n}}$ for some $\delta=\delta(a,\varepsilon)>0$ for every $a,\vep>0$.
\qed


\subsection{The graph is sparse: Proof of Theorem \ref{thm-conc-edges}}
\label{sec:sparse}

We will use the technique of generating functions to show that the graph is sparse. We need the following standard lemma:

\begin{lem}[Concentration in terms of moment generating functions]
\label{lemma:tuedec20434pm}
Let $X_n$ be a random variable such that
	\begin{equation}
	\label{eqn:tuedec20456pm}
	\lim_{n\rightarrow \infty}\frac{1}{n}\log \mathbb{E}\left[\e^{tX_n}\right] = \phi(t),
	\qquad t \in \mathbb{R},
	\end{equation}
for some $\phi\colon\,\mathbb{R}\rightarrow \mathbb{R}$ that is twice differentiable at $0$ and satisfies $\phi(0)=0$ and $\sup_{t \in [0,\delta]}|\phi''(t)|\leq C_{\delta}<\infty$ for some $\delta>0$. Then $\frac{X_n}{n}\rightarrow \phi'(0)$ in probability as $n\to\infty$.
\end{lem}

\begin{proof} 
We use the Markov inequality. Fix any $\varepsilon>0$. Then, for all $t>0$,
	\begin{equation}
	\label{eqntuedec20454pm}
	\begin{aligned}
	&\mathbb{P}\left(\frac{X_n}{n}\geq \phi'(0)+\varepsilon\right) 
	= \mathbb{P}\left(t\frac{X_n}{n}\geq t\phi'(0)+t\varepsilon\right) \\
	&=\mathbb{P}\left(\exp(tX_n)\geq\exp( tn\phi'(0)+ tn\varepsilon)\right) \\
	&\leq \exp(-tn\phi'(0)-tn\varepsilon)\,\mathbb{E}\left[\e^{tX_n}\right].
	\end{aligned}
	\end{equation}
By \eqref{eqn:tuedec20456pm},
	\begin{equation}
	\limsup_{n\rightarrow \infty}\frac{1}{n}\log\mathbb{P}\left(\frac{X_n}{n}\geq \phi'(0)+\varepsilon\right)\leq -t\phi'(0)-t\varepsilon + \phi(t). 
	\end{equation}
Expanding $\phi$, we get, for $t\in [0,\delta]$,
	\begin{equation}
	\limsup_{n\rightarrow \infty}\frac{1}{n}\log\mathbb{P}\left(\frac{X_n}{n}\geq \phi'(0)+\varepsilon\right)
	\leq -t\varepsilon + {\tfrac12} C_{\delta}t^2. 
	\end{equation}
Now we can choose $t$ small enough to complete the proof of the upper bound. The lower bound works in the same way. 
\end{proof}

The following lemma gives us control over the number of edges in $G_n$ given $V_T(G_n)$:

\begin{lem}[Conditional edge moment generating function]
\label{lem:tuedec201245pm}
For every $a \in (0,1]$ and $t\in \mathbb{R}$,
	\begin{equation}
	\lim_{n\rightarrow \infty} \frac{1}{n} \log\mathbb{E}\big[\e^{te(G_n)}~|~V_T(G_n)=an\big]
	= \tfrac{1}{2}\lambda (\e^t-1)+at.
	\end{equation}
\end{lem}

\begin{proof} 
For fixed $t \in \mathbb{R}$, compute 
	\begin{equation}
	\begin{aligned}
	\mathbb{E}\left[\e^{te(G_n)}\indic{V_T(G_n)=an}\right]
	= \sum_{G\colon V_T(G)=an} \e^{te(G)}\, p_n^{e(G)}(1-p_n)^{{n\choose 2}-e(G)}.
	\end{aligned}
	\end{equation}
We can rewrite this relation as 
	\begin{equation}
	\begin{aligned}
	\mathbb{E}\big[\e^{te(G_n)}\indic{V_T(G_n)=an}\big] = \left(\phi_n(t)\right)^{{n\choose 2}}\
	\sum_{G\colon V_T(G)=an}  {p_n(t)}^{e(G)}(1-p_n(t))^{{n\choose 2}-e(G)},
	\end{aligned}
	\end{equation}
where 
	\begin{equation}
	\phi_n(t) = 1-p_n+\e^tp_n, \qquad p_n(t)=\frac{\e^tp_n}{1-p_n+\e^tp_n}.
	\end{equation}
Hence
	\begin{equation}
	\label{eqn:tuedev201115am}
	\mathbb{E}\big[\e^{te(G_n)}\indic{V_T(G_n)=an}\big]
	= \left(\phi_n(t)\right)^{{n\choose 2}} \mathbb{P}_{p_n(t)}(V_T(G_n)=an).
	\end{equation}
Using \eqref{eqn:tuedev201115am}, we get
	\begin{equation}
	\label{eqn:tuedec201155am}
	\mathbb{E}\big[\e^{te(G_n)}~|~V_T(G_n)=an\big]
	= \frac{\mathbb{E}\left[\exp(te(G_n))\indic{V_T(G_n)=an}\right]}{\mathbb{P}(V_T(G_n)=an)} 
	= \left(\phi_n(t)\right)^{{n\choose 2}}\frac{\mathbb{P}_{p_n(t)}(V_T(G_n)=an)}{\mathbb{P}(V_T(G_n)=an)}.
	\end{equation}
At this point we use the estimate from \eqref{eqn:thudec221128pm}, to get
	\begin{equation}
	\label{eqn:extra}
	\left(\frac{p_n(t)}{p_n}\right)^{an-(an)^c}\leq \frac{\mathbb{P}_{p_n(t)}(V_T(G_n)=an)}{\mathbb{P}(V_T(G_n)=an)}
	\leq \left(\frac{p_n(t)}{p_n}\right)^{an+(an)^{9/10}}.
	\end{equation}
Combining \eqref{eqn:tuedec201155am}--\eqref{eqn:extra}, we get the claim.
\end{proof}

\begin{proof}[Proof of Theorem \ref{thm-conc-edges}]
We use Lemma \ref{lemma:tuedec20434pm} with $X_n = e(G_n)$ conditionally on $V_T(G_n)=an$. In this case $\phi(t)=\frac{1}{2}\lambda (\e^t-1)+at$, which satisfies the assumptions in Lemma \ref{lemma:tuedec20434pm}.
\end{proof}

From Lemma \ref{lemma:tuedec20434pm}, we can also conclude that the number of edges, conditionally on $V_T(G_n)=an$, satisfied a large deviation principle:
\begin{cor}[LDP for the number of edges]
\label{cor-LDP-edges}
Conditionally on $V_T(G_n)=an$, the number of edges $e(G_n)$ satisfies a large deviation principle with rate $n$ and with good rate function
	\eqn{
	I_\lambda(x)=(x-a)\log\Big(\frac{x-a}{\lambda/2}\Big)-(x-a)+1.
	}
\end{cor}
The rate function $x\mapsto I_\lambda(x)$ is the large deviation rate function of $na+e(G_n)$, where now $G_n$ is an {\em unconditional} \erdos{}. This suggests that the conditioning on $V_T(G_n)=an$ simply adds $\tfrac13 an$ vertex-disjoint triangles, and the remainder of the graph still has that every edge is present independently with probability $p_n=\lambda/n$.

\proof 
The claim follows directly from the G\"artner--Ellis Theorem (see \cite[Theorem 2.3.6]{DemZei98}).
\qed


\section{Exponential random graph models}
\label{sec-ERGs}

In this section we consider exponential random graph models based on the number of vertices in triangles. In Section~\ref{sec-res-ERGs}, we discuss our main results, which are of two types. In the first, we bias by {$\log n$} times the number of vertices in triangles (a setting which we call ``linear tilting''). In the second, we bias by $n \log n$ times a function of the number of vertices in triangles (a setting which we call ``functional tilting''). We investigate the asymptotics of the partition function, as well as the number of vertices in triangles and the number of edges, arising in the exponential random graph. In Section \ref{sec:exptriangles1}, we prove our main results for the linear tilting; in Section~\ref{sec:exptriangles2} for the functional tilting.


\subsection{Results for exponential random graphs}
\label{sec-res-ERGs}


\subsubsection{Linear tilting}

Consider the exponential random graph model $\mathbb{P}_\beta$ defined by 
	\begin{equation}
	\label{eqn:250pm15feb22}
	\frac{d\mathbb{P}_\beta}{d\mathbb{P}} = \frac{1}{Z_n(\beta)}\, \mathrm{e}^{\beta \log{n} V_{T}(G_n)},
	\end{equation}
where $\mathbb{P}$ is the measure of the \erdos{} with $p_n=\lambda/n$. This model was investigated in \cite{ChaHofHol21}, where it was proved that 
$$
V_T(G_n)/n\convpbeta 0 \text{ for } \beta<\tfrac{1}{3}, \qquad
V_T(G_n)/n\convpbeta 1 \text{ for } \beta>\tfrac{1}{3}.
$$ 
We investigate the critical window of this phase transition, for which we take
	\eqn{
	\label{beta-choice-linear}
	\beta=\frac{1}{3}+\frac{\theta}{\log n}.
	}
It turns out that for this choice of parameter, the number of vertices in triangles concentrates on a non-trivial value $a^\star \in (0,1)$ that depends on $\theta$:

\begin{thm}[Vertex-in-triangles exponential random graphs]
\label{thm-linear-tilting-ERG}
Consider the exponential random graph in \eqref{eqn:250pm15feb22}, with $\beta=\tfrac13+\frac{\theta}{\log n}$. Then, for all $\theta\in {\mathbb R}$,
	\begin{equation}
	\label{universality:ergm-linear}
	\lim_{n\rightarrow \infty} \frac{1}{n}\log{Z_n(\beta)} = \max_{a\in[0,1]} \Lambda(a,\theta,\lambda),
	\end{equation}
with 
	\begin{equation}
	\Lambda(a,\theta,\lambda) = \theta a - (1-a) \log(1-a) - (\tfrac13 a) \log(\tfrac13 a)
	-\tfrac23 a -\tfrac13 a \log{6} + a\log{\lambda}.
	\end{equation}
Furthermore,
	\begin{equation}
	\label{eqn:wed20303pm}
	\frac{V_T(G_n)}{n}\convpbeta a^\star,
	\end{equation}
where $a^\star$ is the unique maximizer of the variational problem in \eqref{universality:ergm-linear}, and
	\begin{equation}
	\label{eqn:wed20304pm}
	\frac{E(G_n)}{n}\convpbeta \frac{\lambda}{2}+a^\star.
	\end{equation}
\end{thm}

\subsubsection{Functional tilting}
In this section, we extend Theorem \ref{thm-linear-tilting-ERG} to exponential random graphs based on the number of vertices in triangles. Pick a function $g\colon\,[0,1]\rightarrow \mathbb{R}$. Consider the exponential random graph model $\mathbb{P}_g$ defined by 
	\begin{equation}
	\label{eqn:250pm30mar22}
	\frac{d\mathbb{P}_g}{d\mathbb{P}} = \frac{1}{Z_n(g)}\, \mathrm{e}^{ n\log{n} g\left(V_{ T}(G_n)/n\right)},
	\end{equation}
where $\mathbb{P}$ is as before. We show that, under some mild assumptions on $g$, we can asymptotically evaluate the normalizing constant $Z_n(g)$, as well as the number of vertices in triangles and the number of edges:

\begin{thm}[Functional exponential random graphs]
\label{thm-func-tilting-ERG}
Let $g\colon\,[0,1]\rightarrow \mathbb{R}$ be a function such that all maxima of $a\mapsto g(a)-\tfrac13 a$ are attained in $(0,1)$. Let $a^\star \in (0,1)$ be any of the maximizers. Then
	\begin{equation}
	\label{universality:ergm}
	\lim_{n\rightarrow \infty} \frac{1}{n\log{n}}\log{Z_n(g)} = g(a^\star)-\tfrac13 a^\star.
	\end{equation}
Further, if the maximum of $g(a)-\tfrac13 a$ is {\em uniquely} attained at some $a^\star \in (0,1)$, then
	\begin{equation}
	\label{eqn:weddec20332pm}
	\frac{V_T(G_n)}{n}\convpg a^\star,
	\end{equation}
while
	\begin{equation}
	\label{eqn:weddec20333pm}
	\frac{E(G_n)}{n}\convpg a^\star+\frac{\lambda}{2}.
	\end{equation}
\end{thm} 

Here is a concrete example of a function $g$ for which Theorem \ref{thm-func-tilting-ERG} applies.

\begin{cor}[Example of functional tilting]
\label{cor:weddec20308pm}
Let $g(x)= \beta x^{\alpha}$ for some $\alpha \in (0,1)$, $\beta>0$, and $3\alpha \beta < 1$. Then
\begin{equation}
\label{eqn:thumar30107pm}
\lim_{n\rightarrow \infty}\frac{1}{n}\log{Z_n(g)}
= \beta (3\alpha \beta)^{-\frac{\alpha}{\alpha-1}} - \tfrac{1}{3}(3\alpha \beta)^{-\frac{1}{\alpha-1}}.
\end{equation}
\end{cor}

\begin{proof}
Let $G(x)= g(x)-\tfrac13 x$. Note that $G^{'}(x)= \beta \alpha x^{\alpha -1} -\frac{1}{3}$ and $G^{''}(x)= \beta \alpha (\alpha-1) x^{\alpha -2} $. Therefore $G$ attains its maximum in $(0,1)$ (since $3\alpha \beta <1$). Therefore the claim follows from Theorem \ref{thm-func-tilting-ERG}.
\end{proof}

\begin{rmk}[LDP for the number of edges]
\label{rem-LDP-edges}
{\rm In Theorems \ref{thm-linear-tilting-ERG} and \ref{thm-func-tilting-ERG}, we can extend the asymptotics of the number of edges to a large deviation principle, as in Corollary \ref{cor-LDP-edges}. The proof of this extension follows directly from the G\"artner--Ellis Theorem (see \cite[Theorem 2.3.6]{DemZei98}), and is therefore identical to the proof of Corollary \ref{cor-LDP-edges}.}
\end{rmk}

In the remainder of the section, we give the proofs of Theorems \ref{thm-linear-tilting-ERG} and \ref{thm-func-tilting-ERG}. The proofs are organised in the same way. First, we investigate the asymptotics of the partition function in terms of a variational problem. Afterwards, we conclude that the number of vertices in triangles divided by $n$ converges to the maximizer of this variational problem. Finally, we compute the moment generating function for the number of edges to conclude the proof.


\subsection{Linear tilting}
\label{sec:exptriangles1}

The asymptotics for partition function is identified in the following lemma:

\begin{lem}[Partition function for the linear tilting]
For every $\theta \in \mathbb{R}$,
	\begin{equation}
	\label{eqn:satoct71210pm}
	  \log Z_n\big(\tfrac{1}{3}+\tfrac{\theta}{\log{n}}\big) 
	 = n\max_{0 \leq a \leq 1} \Lambda(a,\theta,\lambda) +o(n^{19/20}).
	\end{equation}
\end{lem}

\begin{proof} 
By Theorem \ref{thm:ldp_vertices_in_triangle}, 
	\begin{equation}
	\label{eqn:606pm05apr222}
	\begin{aligned}
	&\log{\mathbb{P}(V_{ T}(G_n) = an)} \\
	& = -n(1-a) \log(1-a) - \tfrac13 an \log{n}- \tfrac13 an \log(\tfrac13a) 
	- \tfrac23 an -\tfrac13 an \log{6}+an\log{\lambda} +o(n^{19/20}).
	\end{aligned}
	\end{equation}
We write
\begin{equation*}
\mathbb{E}\left(\e^{\beta \log{n} V_{T}(G_n)}\right) = \sum_{q=0}^n \e^{\beta \log{n} q} \,
\mathbb{P}\left(V_{T}(G_n)=q\right).
\end{equation*}
For $p_n = \lambda/n$, using the estimate in \eqref{eqn:606pm05apr222}, we obtain 
\begin{equation}
\label{eqn:wednov22602pm}
\begin{aligned}
&\mathbb{E}\left(\e^{\beta \log{n} V_{T}(G_n)}\right) =  n\sup_{0\leq a\leq 1}\\
&\times \exp\left({\beta an\log{n} -n(1-a) \log(1-a) - \tfrac13 an \log(\tfrac13 an) 
- an (\tfrac23  +\tfrac13 \log{6}- \log{\lambda}) + o(n^{19/20})}\right).
\end{aligned}
\end{equation} 
Substitute $\beta= \tfrac{1}{3}+\frac{\theta}{\log{n}}$, to get
\begin{equation}
\mathbb{E}\left(\e^{\beta \log{n} V_{T}(G_n)}\right) 
= n \exp\left(n\max_{0\leq a\leq 1}\Lambda(a,\theta,\lambda) + o(n^{19/20})\right).
\end{equation}
\end{proof}

\begin{proof}[Proof of Theorem \ref{thm-linear-tilting-ERG}.] Fix an $\vep>0$. We wish to upper bound $\mathbb{P}_\beta(V_{T}(G_n) \not\in [(a^\star-\varepsilon)n, (a^\star+\varepsilon)n])$. This can be written as
\begin{equation}
\mathbb{P}_\beta\left(V_{T}(G_n) \not\in [(a^\star-\varepsilon)n, (a^\star+\varepsilon)n]\right) 
= \frac{1}{Z_n\big(\beta\big) }\sum_{q\not\in [(a^\star-\varepsilon)n, (a^\star+\varepsilon)n]} \mathrm{e}^{\beta \log{n} q} 
\,\, \mathbb{P}(V_{T}(G_n) =q).
\end{equation}
Define
\begin{equation}
\delta = \Lambda(a^\star,\theta,\lambda) - \max_{a\not\in [(a^\star-\varepsilon), (a^\star+\varepsilon)]} \Lambda(a,\theta,\lambda).
\end{equation}
It is easy to check that $\delta>0$, since $\Lambda(\cdot,\theta,\lambda)$ is strictly concave and therefore has a unique maximum. Now we set $\beta= \tfrac{1}{3}+\frac{\theta}{\log{n}}$ and use \eqref{eqn:606pm05apr222}, to obtain that $\log \sum_{q\not\in [(a^\star-\varepsilon)n, (a^\star+\varepsilon)n]} \mathrm{e}^{\beta \log{n} q} \,\, \mathbb{P}(V_{T}(G_n) =q)$ is bounded above by $n(\tfrac14\delta+ \max_{a\not\in [(a^\star-\varepsilon), (a^\star+\varepsilon)]} \Lambda(a,\theta,\lambda))$ for large enough $n$. Then use $\eqref{eqn:satoct71210pm}$ to get $\log Z_n(\tfrac{1}{3}+\tfrac{\theta}{\log{n}}) \geq n(-\tfrac14\delta+ \Lambda(a^\star,\theta,\lambda))$. Combining these estimates, we get
\begin{equation}
\log \mathbb{P}_{(\tfrac{1}{3}+\frac{\theta}{\log{n}})}\left(V_{T}(G_n) 
\not\in [(a^\star-\varepsilon)n, (a^\star+\varepsilon)n]\right) \leq - \tfrac12 n\delta.
\end{equation}

Let us next compute the following moment generating function for the number of edges under the measure $\mathbb{P}_\beta$ defined in \eqref{eqn:250pm15feb22}. For $t \in \mathbb{R}$,
\begin{equation}
\label{eqn:thunov9117pm}
\mathbb{E}_\beta\left[\e^{te(G_n)}\right]= \frac{1}{Z_n\big(\beta\big)} \mathbb{E}\left(\e^{te(G_n)+ \beta \log{n} V_{T}(G_n)}\right)
= \frac{\left(\phi_n(t)\right)^{{n\choose 2}}}{Z_n\big(\beta\big)} \mathbb{E}_{p_n(t)}\left(\e^{\beta \log{n} V_{T}(G_n))}\right),
\end{equation}
where $\mathbb{E}_{p_n(t)}$ is the expectation w.r.t.\ the \erdos{} with $p=p_n(t)$.
 
Now, using \eqref{eqn:satoct71210pm}, we get 
\begin{equation}
\label{eqn:wednov22629pm}
\begin{aligned}
&\left(\phi_n(t)\right)^{{n\choose 2}} \exp{\left(n\max_{0 \leq a \leq 1} \Lambda(a,\theta,np_n(t)) 
- n\max_{0 \leq a \leq 1} \Lambda(a,\theta,\lambda) +o(n^{19/20})\right)} \leq \mathbb{E}_\beta\left[\e^{te(G_n)}\right] \\
&\leq \left(\phi_n(t)\right)^{{n\choose 2}} \exp{\left(n\max_{0 \leq a \leq 1} \Lambda(a,\theta,np_n(t)) 
- n\max_{0 \leq a \leq 1} \Lambda(a,\theta,\lambda) +o(n^{19/20})\right)}.
\end{aligned}
\end{equation}
Finally, simplifying \eqref{eqn:wednov22629pm}, we obtain
$$
\lim_{n \rightarrow \infty} \frac{1}{n} \log{\mathbb{E}_\beta\left[\e^{te(G_n)}\right]} = \tfrac{1}{2}\lambda (\e^t-1) 
+ \max_{0 \leq a \leq 1} \Lambda(a,\theta+t,\lambda)-\max_{0 \leq a \leq 1} \Lambda(a,\theta,\lambda).
$$
Differentiating the right-hand side of the last display with respect to $t$ at $t=0$, we find $\lambda/2+a^{\star}$. (The first term is easy while the second term is handled in Lemma \ref{techlemma:weddec20200pm} below.) Therefore 
$$
\frac{E(G_n)}{n}\convpbeta \frac{\lambda}{2}+a^{\star}
$$ 
by Lemma \ref{lemma:tuedec20434pm}.
\end{proof}


\begin{lem}
\label{techlemma:weddec20200pm}

Let $\Lambda_{\max}(t)=\max_{0\leq a\leq 1}\Lambda(a,\theta+t,\lambda)$. Then
$$
\frac{\partial\Lambda_{\max}(t)}{\partial t}  \big |_{t=0} = a^{\star}.
$$
\end{lem}
\begin{proof}
Recall that 
$$
\Lambda(a,\theta,\lambda) = \theta a - (1-a) \log(1-a) - (\tfrac13 a) \log(\tfrac13 a)-\tfrac23 a -\tfrac13 a \log{6} + a\log{\lambda}.
$$
It is easy to see that $\Lambda$ is concave in $a$ and thus has a unique maximum. Define 
$$
a^\star(t) = \arg\max_{0\leq a \leq 1} \Lambda(a,\theta+t,\lambda).
$$ 
Also note that $a^\star(t)$ must be the unique solution of the equation $\frac{\partial \Lambda}{\partial a}=0$. More precisely,
\begin{equation}
\label{eqn:weddec201227pm}
\theta+t+\log(1-a^\star(t))-\tfrac{1}{3}\log(\tfrac{1}{3}a^\star(t))-\tfrac13 \log{6} + \log{\lambda}=0.
\end{equation}
Clearly $a^\star(0)=a^\star$, and 
\begin{equation}
\label{eqn:weddec201247pm}
\Lambda_{\max}(t)=\max_{0\leq a\leq 1}\Lambda(a,\theta+t,\lambda)= \Lambda(a^\star(t),\theta+t,\lambda).
\end{equation}
We can implicitly differentiate $a^{\star}(t)$ with respect to $t$ and use \eqref{eqn:weddec201227pm}, to obtain
\begin{equation}
\frac{\partial a^{\star}(t)}{\partial t} = \frac{3a^{\star}(t)(1-a^{\star})}{1+2a^{\star}}.
\end{equation}
We are now ready to evaluate $\frac{\partial\Lambda_{\max}(t)}{\partial t}$. Using \eqref{eqn:weddec201247pm}, we get 
\begin{equation}
\label{eqn:weddec201255}
\begin{aligned}
\frac{\partial\Lambda_{\max}(t)}{\partial t} 
&= \frac{\partial\Lambda(a,\theta+t,\lambda)}{\partial a}\Big|_{a=a^\star(t)} 
\times \frac{\partial a^{\star}(t)}{\partial t} +\frac{\partial\Lambda(a,\theta+t,\lambda)}{\partial \theta}\Big|_{a=a^\star(t)}\\
&= \Big(\theta+t +\log(1-a^\star(t))-\tfrac{1}{3}\log(\tfrac{1}{3}a^\star(t))-\tfrac13 \log{6} + \log{\lambda}\Big)
\frac{3a^{\star}(t)(1-a^{\star}(t))}{1+2a^{\star}(t)}+a^{\star}(t).
\end{aligned}
\end{equation}
Therefore \eqref{eqn:weddec201255} yields $\frac{\partial\Lambda_{\max}(t)}{\partial t}  \big |_{t=0} = a^{\star}(0)=a^{\star}$.
\end{proof}


\subsection{Functional tilting}
\label{sec:exptriangles2}
The following lemma identifies the asymptotics of the partition sum:
 
\begin{lem}[Partition function for the functional tilting]
\label{prop:thumar301259pm}
Let $g\colon\,[0,1]\rightarrow \mathbb{R}$ be a function such that all maxima of $a\mapsto g(a)-\tfrac13 a$ are attained in $(0,1)$. Let $a^\star \in (0,1)$ be any of the maximizers. Then
	\begin{equation}
	\label{universality:ergm-alt}
	\lim_{n\rightarrow \infty} \frac{1}{n\log{n}}\log{Z_n(g)} = g(a^\star)-\tfrac13 a^\star,
	\end{equation}
and $V_{ T}(G_n)/n\convpg a^\star$. 
\end{lem}

\begin{proof}
Using \eqref{eqn:thudec221128pm}, we get 
	\begin{equation}
	\label{eqn:314pm15mar22}
	\begin{aligned}
	\mathbb{E}\left(\mathrm{e}^{n\log{n}\,g(\frac{V_{ T}(G_n)}{n})}\right) 
	&= \sum_{q=0}^n \mathrm{e}^{n\log{n} g(\frac{q}{n})}\exp\big(-\tfrac13 q\log{n} +O(n)\big)\\
	&\leq n \exp\Big(n\log{n}\,\sup_{a\in [0,1]} (g(a)-\tfrac13 a) +O(n)\Big).
	\end{aligned}
	\end{equation}
By assumption, the supremum is attained at $a^\star \in(0,1)$. Therefore we obtain the following lower bound as well:
	\begin{equation}
	\label{eqn:thumar301243pm}
	\mathbb{E}\left(\mathrm{e}^{n\log{n} g(\frac{V_{ T}(G_n)}{n})}\right) 
	\geq \exp\Big(n\log{n}\,(g(a^\star)-\tfrac13 a^\star) +O(n)\Big).
	\end{equation}
Also, the main contribution comes from $V_{ T}(G_n)\in [(a^\star-\varepsilon)n, (a^\star+\varepsilon)n]$. Combining \eqref{eqn:314pm15mar22}--\eqref{eqn:thumar301243pm}, we get the claim.
\end{proof}

The limit of the number of edges is slightly more involved. The next lemma identifies its moment generating function, as in Lemma \ref{lem:tuedec201245pm}:

\begin{lem}[Edge moment generating function for functional tilting]
\label{lem-edge-MGF-functional}
Let $g\colon\,[0,1]\rightarrow \mathbb{R}$ be any function such that $a\mapsto g(a)-\tfrac13 a$ is uniquely maximized at $a^\star \in (0,1)$. Let $\mathbb{E}_g$ be the expectation with respect to $\mathbb{P}_{g}$ defined in \eqref{eqn:250pm30mar22}. Then, for any $t \in \mathbb{R}$, 
	\begin{equation}
	\lim_{n\rightarrow \infty}\frac{1}{n}\log{\mathbb{E}_g\left[\e^{te(G_n)}\right]} = \tfrac{1}{2}\lambda (\e^t-1)+a^\star t.
	\end{equation}
\end{lem}

\begin{proof} We wish to compute 
	\begin{equation}
	\mathbb{E}_g\left[\e^{te(G_n)}\right] = \frac{1}{Z_n(g)}\, 
	\mathbb{E}\left(\exp{\big(te(G_n)+ n\log{n}\, g(\tfrac{V_{ T}(G_n)}{n})\big)}\right). 
	\end{equation}
This relation can be rewritten as 
	\begin{equation}
	\mathbb{E}_g\left[\e^{te(G_n)}\right] = \left(\phi_n(t)\right)^{{n\choose 2}} \frac{1}{Z_n(g)}\, 
	\mathbb{E}_{p_n(t)}\left(\exp{\big( n\log{n}\, g(\tfrac{V_{ T}(G_n)}{n})\big)}\right). 
	\end{equation}
Write
	\begin{equation}
	\mathbb{E}_{p_n(t)}\left(\exp{\big(n\log{n}\, g(\tfrac{V_{ T}(G_n)}{n})\big)}\right) 
	= \sum_{q=0}^n \mathrm{e}^{n\log{n} g\left(\frac{q}{n}\right)}\, \mathbb{P}_{p_n(t)}(V_{T}(G_n)=q).
	\end{equation}
Using the estimate in \eqref{eqn:monfeb21427pm}, we obtain that the sum is bounded above by
	\begin{equation}
	\begin{aligned}
	&n \sup_{0\leq a \leq 1}\exp\Big(n\log{n}\,(g(a)- \tfrac13a) -n(1-a) \log(1-a) - (\tfrac13an)\log(\tfrac13a)\\
	&\qquad\qquad\qquad -\tfrac23 an -\tfrac13 an\log{6}+an\log{(\e^t \lambda)} +o(n^{19/20})\Big).
	\end{aligned}
	\end{equation}
	
In order to study this function, we need the following technical lemma:

\begin{lem}
\label{lem:may31121am}
Let $f_1\colon\,[0,1]\rightarrow \mathbb{R}$ and $f_2\colon\,[0,1]\rightarrow \mathbb{R}$ be two functions such that $f_1$ uniquely attains its maximum at $a^\star \in [0,1]$ and $f_1, f_2$ are continuous on $[0,1]$. Let $(a_n)_{n\in\mathbb{N}}$,  $(b_n)_{n\in\mathbb{N}}$ be sequences of non-negative real numbers such that $b_n/a_n \rightarrow 0$ as $n\to \infty$. Then, for every fixed $\varepsilon>0$, $a_nf_1(a)+b_nf_2(a) \leq a_nf_1(a^\star)+b_nf_2(a^\star) +\varepsilon b_n$ for sufficiently large $n$ and all $a\in(0,1)$.
\end{lem}

\begin{proof} 
Fix a $\delta>0$. First consider the situation when $|a-a^\star|\geq \delta$. Let the minimum value of $f_1(a^\star) - f_1(a)$ on the compact set $\{|a-a^\star|\geq \delta\} \cap [0,1]$ be $D>0$ ($D$ must be positive because $f_1$ has a unique maximum), and the minimum value of $f_2(a^\star) - f_2(a)$ on $[0,1]$ must be $M$. Then
	\begin{equation}
	\begin{aligned}
 	a_nf_1(a^\star)+b_nf_2(a^\star) - a_nf_1(a)-b_nf_2(a) 
 	&= a_n\Big(f_1(a^\star) - f_1(a) + \frac{b_n}{a_n}(f_2(a^\star) - f_2(a))\Big)\\
 	&\geq a_n\Big(D + \frac{b_n}{a_n}M\Big) \geq 0
	\end{aligned}
	\end{equation}
for large enough $n$. If $|a-a^\star|< \delta$, then we use $f_1(a) \leq f_1(a^\star)$ and continuity of $f_2$, and choose $\delta$ small enough. 
\end{proof}

Since $a\mapsto g(a)-\tfrac13 a$ is uniquely maximized at $a^\star$, we can use Lemma \ref{lem:may31121am} with $a_n=n\log{n}$ and $b_n=n$ to obtain the following upper bound, valid for large enough $n$:
	\begin{equation}
	\label{e:form5}
	\begin{aligned}
	&n \exp\Big(n\log{n}\,(g(a^\star)- \tfrac13a^\star +\varepsilon) - n(1-a^\star) \log(1-a^\star) 
	- (\tfrac13 a^\star n) \log(\tfrac13a^\star n)\\
	&\qquad \qquad -\tfrac23 a^\star n -\tfrac13a^\star n\log{6}+a^\star n\log{(\e^t \lambda)} +o(n)\Big).
	\end{aligned}
	\end{equation}
Obtaining a lower bound is easy, namely,
	\begin{equation}
	\mathbb{E}_{p_n(t)}\left[\e^{n\log{n}\,g(\tfrac{V_{ T}(G_n)}{n})}\right]
	\geq  \mathrm{e}^{n\log{n}\, g\left(a^\star\right)}\,\mathbb{P}_{p_n(t)}(V_{T}(G_n)=a^\star n).
	\end{equation}
We use the estimates in Theorem~\ref{thm:ldp_vertices_in_triangle} to obtain the following lower bound for the exponential in \eqref{e:form5}:
	\begin{equation}
	\begin{aligned}
	&\exp\Big(n\log{n}\,(g(a^\star)- \tfrac13a^\star) -n(1-a^\star) \log(1-a^\star) - (\tfrac13a^\star)n\log(\tfrac13a^\star)\\ 
	&\qquad\qquad  - \tfrac23 a^\star n - \tfrac13 a^\star n \log{6} + a^\star n\log{\lambda} +o(n)\Big).
	\end{aligned}
	\end{equation}
Therefore we obtain the sandwich 
	\begin{equation}
	\begin{aligned}
	\frac{1}{n} \left(\phi_n(t)\right)^{{n\choose 2}}\left(\frac{p_n(t)}{p_n}\right)^{a^\star n} \e^{-\varepsilon n}
	&\leq \frac{1}{Z_n(g)}\, \mathbb{E}_{p_n(t)}\left(\exp{\big(n\log{n}\, g(\tfrac{V_{ T}(G_n)}{n})\big)}\right)\\ 
	&\leq n \left(\phi_n(t)\right)^{{n\choose 2}} 
	\left(\frac{p_n(t)}{p_n}\right)^{a^\star n} \e^{\varepsilon n},
	\end{aligned}
	\end{equation}
from which the claim follows.
\end{proof}


\section{Consistent parameter estimation in exponential random graphs}
\label{sec-consistent-estimation-ERGs}

In practice, we can only observe a large network without knowing its full architecture. From the modeling perspective it is important to be able to estimate unknown parameter(s) from observations. In this section, we show that it is possible to consistently estimate the parameters in the exponential random graph in \eqref{eqn:250pm15feb22}, with $\beta=\tfrac13+\frac{\theta}{\log n}$. Note that the distribution of the graph is characterized by two parameters: $\theta$ and $\lambda$. The estimation procedure is a by product of Theorem \ref{thm-linear-tilting-ERG}, and is stated in the following theorem:

\begin{thm}
Let $G_n$ be an observation from the exponential random graph in \eqref{eqn:250pm15feb22}, with $\beta=\tfrac13+\frac{\theta}{\log n}$. Define 
$$
\hat{\lambda}_n = \frac{2E(G_n)}{n} -\frac{2V_T(G_n)}{n}
$$ 
and 
$$
\hat{\theta}_n = \{x\colon\, \arg\max_{0\leq a\leq 1}\Lambda(a,x,\hat{\lambda}_n)= V_T(G_n)/n\}.
$$
Then $\hat{\lambda}_n \convpbeta \lambda$ and $\hat{\theta}_n \convpbeta \theta$ as $n \to \infty$.
\end{thm}

\begin{proof}
The proof follows from \eqref{eqn:wed20303pm}--\eqref{eqn:wed20304pm} and the continuous mapping theorem.
\end{proof}

\begin{rmk}
{\rm It is not hard to come up with models where consistent estimation is possible via Theorem \ref{thm-func-tilting-ERG}. As a proof of concept, let us consider the example in Corollary \ref{cor:weddec20308pm}. In this case, assume that $\alpha$ is known, and $\beta$ and $\lambda$ are unknown parameters with the restriction $3\alpha\beta >1$. In the notation of Theorem \ref{thm-func-tilting-ERG}, 
$$
a^\star = \left(\frac{1}{3\alpha \beta}\right)^{\tfrac{1}{1-\alpha}}.
$$ 
Therefore a natural proposal for estimators of $\beta$ and $\lambda$ is the solution of the equations
\begin{equation}
\left(\frac{1}{3\alpha \hat{\beta}_n}\right)^{\tfrac{1}{1-\alpha}} = \frac{V_T(G_n)}{n}, \qquad
\left(\frac{1}{3\alpha \hat{\beta}_n}\right)^{\tfrac{1}{1-\alpha}} + \frac{\hat{\lambda}_n}{2} =\frac{2E(G_n)}{n}.
\end{equation}
Using \eqref{eqn:weddec20332pm}--\eqref{eqn:weddec20333pm} and the continuous mapping theorem, we have $\hat{\beta}_n \to \beta$ and $\hat{\lambda}_n \to \lambda$ as $n\to \infty$ under the measure described in Corollary \ref{cor:weddec20308pm}.}
\end{rmk}


\section{Discussion and open problems}
\label{sec-disc}

In Section~\ref{sec:disc} we discuss our main results, in Section~\ref{sec:prob} we list some open problems.
\medskip 


\subsection{Discussion main results}
\label{sec:disc}
 
It is crucial that our main technical theorem, Theorem \ref{thm:ldp_vertices_in_triangle}, identifies the {\em second-order asymptotics} of the log of the large deviation probability that $V_T(G_n)\geq an$. The first-order asymptotics, of order $n\log n $, was already identified in \cite{ChaHofHol21}. The fact that we can also identify the second-order asymptotics of order $n$ allows us to prove a large deviation principle for the number of edges in the graph (in Theorem \ref{thm-conc-edges}), as well as prove that most triangles are actually vertex disjoint (in Theorem \ref{structure:disjoint triangles}), which would not have been possible with the first-order result only. This is a reflection of the fact that the key large deviation rate is $n$, not $n\log n$, as one might have conjectured after \cite{ChaHofHol21}. 

The above results in turn allowed us to suggest a range of {\em sparse exponential random graph models}, which is important because it is hard to identify sparse exponential random graph models with many triangles. Finally, our results allowed us to prove that the parameters of the model can be consistently estimated, a property that is relevant in practice.


\subsection{Open problems}
\label{sec:prob}
 
Several natural and interesting extensions are possible. The large deviation principle for the number of edges in the graph in Theorem \ref{thm-conc-edges} suggests that also a {\em central limit theorem} should also hold. This is further exemplified by the nice limiting generating function for the number of edges in Lemma \ref{lem:tuedec201245pm}. Unfortunately, while being suggestive, it seems hard to turn this observation into a mathematical proof. Indeed, typically some correlation-type inequality is needed for such proofs, which we do not have at our disposal here.
 
Furthermore, the combination of Theorem \ref{structure:disjoint triangles} (describing that most triangles are disjoint) and Theorem \ref{thm-conc-edges} (describing the convergence of the number of edges) suggests that we might be able to identify the {\em local limit} of the model as well. Indeed, we conjecture that the local limit is exactly the same as that of the model where $\tfrac13 an$ disjoint triangles are randomly dropped inside an \erdos. Unfortunately, it seems difficult to prove such a result. Since the degree distribution is uniformly integrable, the model is {\em tight} in the space of rooted graphs in the local topology, but we do not see how to prove that the limit indicated above really is the local limit.

The type of {\em models} we investigated may be extended as well. We focussed on the number of vertices in triangles, but it would be natural to consider the number of {\em edges in triangles} instead. Since this number can vary much more (the number is at most $n(n-1)/2$ rather than $n$, as for the number of vertices in triangles), it appears to be a significantly more difficult problem. Finally, of course, we could extend the number of parameters in our model, and investigate the behaviour of the associated exponential random graph. In what generality can the parameters still be consistently estimated?

\begin{funding}
The work in this paper was supported by the Netherlands Organisation for Scientific Research (NWO) through Gravitation-grant NETWORKS-024.002.003. {The authors thank the referees for their careful reading of the paper and for suggesting a shorter proof of the lower bound presented in Section \ref{sec:lb}.}
\end{funding}


\bibliographystyle{imsart-nameyear}
\bibliography{references.bib}

\end{document}